\documentclass[10pt]{amsart}
\usepackage{graphicx,color,epsf}
\usepackage [cmtip,arrow]{xy}
\usepackage {pb-diagram,pb-xy} 

\usepackage{amsmath}
\newtheorem{theorem}{Theorem}[section]
\newtheorem{lemma}[theorem]{Lemma}

\newtheorem{proposition}[theorem]{Proposition}

\theoremstyle{definition}
\newtheorem{definition}[theorem]{Definition}
\newtheorem{example}[theorem]{Example}

\theoremstyle{remark}

\numberwithin{equation}{section}

\newcommand {\hide}[1]{}
\newtheorem{notation}{Notation}

\newcommand {\junk}[1]{}
 \newtheorem{algorithm}{\sc Algorithm}

\newcommand {\R} {\mbox{\rm R}}
\newcommand {\s}        {\mbox{\rm sign}}

\newcommand {\Real}[1]   {\mbox{$\mathbb{R}^{#1}$}}
\newcommand {\Sphere}{\mbox{${\bf S}$}}     
 \newcommand {\re}         {\Real{}}

\newcommand {\Q}         {\mathbb{Q}}

\newcommand {\RR} {{\mathcal R}}

\newcommand {\la}   {{\langle}}
\newcommand {\ra}   {{\rangle}}
\newcommand {\eps} {{\varepsilon}}
\newcommand {\E} {{\rm Ext}}
\newcommand {\dist} {{\rm dist}}

\newcommand {\Id} {\mbox{\rm Id}}

\newcommand {\spanof} {{\rm span}}
\newcommand {\Tot} {{\rm Tot}}

\def\addots{\mathinner{\mkern1mu
\raise1pt\vbox{\kern7pt\hbox{.}}
\mkern2mu\raise4pt\hbox{.}\mkern2mu
\raise7pt\hbox{.}\mkern1mu}}

\newcommand{\HH}  {\mbox{\rm H}}
\newcommand{\Ch}  {\mbox{\rm C}}

\newcommand{\NN}  {{\mathcal N}}

\DeclareMathOperator{\diag}{diag}

\begin{document}
\title[]
{Computing the 
Betti numbers of semi-algebraic sets defined by partly quadratic systems
of polynomials}
\author{Saugata Basu}
\address{School of Mathematics,
Georgia Institute of Technology, Atlanta, GA 30332, U.S.A.}
\email{saugata.basu@math.gatech.edu}
\author{Dmitrii V. Pasechnik}
\address{
MAS Division, School of Physical and Mathematical Sciences,
Nanyang Technological University,
21 Nanyang Link,
Singapore 637371.
}
\email{dima@ntu.edu.sg}
\author{Marie-Fran\c{c}oise Roy}
\address{IRMAR (URA CNRS 305), 
Universit\'{e} de Rennes I,
Campus de Beaulieu 35042 Rennes cedex FRANCE.}
\email{marie-francoise.roy@univ-rennes1.fr}

\thanks{The first author was supported in part by an NSF 
grant CCF-0634907. 
The second author was supported
by an NTU startup grant.
Part of this work was done when the authors
were visiting the Institute of
Mathematics and Its Application, Minneapolis, IRMAR (Rennes) and
the Nanyang Technological University, Singapore.} 





\begin{abstract}
Let $\R$ be a real closed field,
$
{\mathcal Q} \subset \R[Y_1,\ldots,Y_\ell,X_1,\ldots,X_k],
$
with
$
\deg_{Y}(Q) \leq 2, \deg_{X}(Q) \leq d, Q \in {\mathcal Q}, \#({\mathcal Q})=m$, 
and
$
{\mathcal P}
\subset \R[X_1,\ldots,X_k]
$
with 
$\deg_{X}(P) \leq d, P \in {\mathcal P}, \#({\mathcal P})=s$.
Let $S \subset \R^{\ell+k}$   be
a semi-algebraic set defined by a 
Boolean formula without negations, with
atoms  $P=0, P \geq 0, P \leq 0, \;
P \in {\mathcal P} \cup {\mathcal Q}$.
We describe an algorithm for computing the 
the Betti numbers of $S$
generalizing a similar algorithm described in \cite{Bas05-top}.
The complexity of the algorithm is bounded by
$(\ell s m d)^{2^{O(m+k)}}$. 
The complexity of the algorithm
interpolates between the doubly exponential
time bounds for the known algorithms in the general case, and the polynomial
complexity in case of semi-algebraic sets defined by few quadratic 
inequalities \cite{Bas05-top}.
Moreover, for fixed $m$ and $k$ this algorithm
has polynomial time complexity in the remaining parameters.
\end{abstract}

\maketitle

\section{Introduction and Main Results}
\label{sec:intro}
Let $\R$ be a real closed field and $S \subset \R^k$ a semi-algebraic set
defined by a
Boolean formula with atoms of the form
$P > 0, P < 0, P=0$ for $P \in {\mathcal P} \subset \R[X_1,\ldots,X_k]$.
We call $S$ a ${\mathcal P}$-semi-algebraic set 
and the Boolean formula defining $S$ a ${\mathcal P}$-formula.
If, instead, the Boolean formula
has atoms of the form $P=0, P \geq 0, P \leq 0, \;P \in {\mathcal P}$, 
and additionally contains no negation,
then we will call $S$ a ${\mathcal P}$-closed semi-algebraic set,
and the formula defining $S$ a ${\mathcal P}$-closed formula. 
Moreover, we call a ${\mathcal P}$-closed semi-algebraic set $S$ basic
if the ${\mathcal P}$-closed formula defining $S$ is a conjunction of
atoms of the form $P=0, P \geq 0, P \leq 0, \;P \in {\mathcal P}$.

For any closed semi-algebraic set $X \subset \R^k$, 
we denote by $b_i(X)$ the dimension  of the
$\Q$-vector space, $\HH_i(X,\Q)$, which is the $i$-th homology group of $X$
with coefficients in $\Q$.  We refer to \cite{BPRbook2} for the definition of
homology in the case of $\R$ being 
an arbitrary real closed field, 
not necessarily the field of real numbers.

\subsection{Brief History}
Designing efficient algorithms of computing the 
Betti numbers of semi-algebraic sets is an
important problem which has been considered by several authors. 
We give a brief description of the recent advances 
and direct the reader to the survey article \cite{Basu_survey} 
for a more detailed exposition.

For general semi-algebraic sets the best known algorithm for computing 
all the Betti numbers is via triangulation using cylindrical algebraic 
decomposition (see for example \cite{BPRbook2}) whose 
complexity is {\em doubly exponential} in the number of variables. 
There have been some small advances in obtaining singly exponential
time algorithms for computing some of the Betti numbers, but we are still
very far from having an algorithm for computing all the Betti numbers
of a given semi-algebraic set in singly exponential time.
Singly exponential time algorithms for computing the number of connected
components of a semi-algebraic set $S$ (i.e. $b_0(S)$) has been known
for quite some time \cite{GV92,Canny93a,GR92,HRS94,BPR99}.
More recently,  an algorithm with singly exponential complexity 
is given  in \cite{BPRbettione}
for computing the first Betti number of  semi-algebraic sets.
The above result is generalized in \cite{Bas05-first}, 
where a singly exponential time algorithm is given 
for computing the first $\ell$ Betti
numbers of semi-algebraic sets, where $\ell$ is allowed to be any
constant. Finally, note that singly exponential
time algorithm is also known for computing the Euler-Poincar\'e characteristic
(which is the alternating sum of Betti numbers) 
of semi-algebraic sets \cite{B99}.

In another direction, several researchers have considered a special
class of semi-algebraic sets -- namely, semi-algebraic sets defined
using quadratic polynomials. While the topology
of such sets can be arbitrarily complicated 
(since any semi-algebraic set can be defined as the image under
a linear projection of a semi-algebraic defined by quadratic 
inequalities),
it is possible to prove bounds on the Betti numbers of such sets 
which are
polynomial in the number of variables and exponential in only the
number of inequalities \cite{Bar97,Bas05-first-Kettner}.
(In contrast a semi-algebraic set defined by
a single polynomial of degree $> 2$ can have exponentially large
Betti numbers.)
Polynomial time algorithms for testing emptiness of such sets
(where the number of inequalities is fixed) 
were given in \cite{Bar93, GrPa04}.
A polynomial time algorithm (without any restriction on the
number of inequalities) is given 
in \cite{Bas05-top}(see also \cite{Bas05-top-errata})
for computing a
constant number of the top Betti numbers of semi-algebraic sets 
defined by quadratic inequalities. If moreover 
the number of inequalities is fixed then
the algorithm computes all the Betti numbers in polynomial time.
More precisely, an algorithm is described which takes as input a 
semi-algebraic set, $S$, defined by 
$Q_1 \geq 0,\ldots,Q_m \geq 0$, where each $Q_i \in \R[Y_1,\ldots,X_\ell]$
has degree $\leq 2,$
and computes the top $p$ Betti numbers of $S$,
$b_{k-1}(S), \ldots, b_{k-p}(S),$ in polynomial time. 
The complexity of the algorithm is
$ 
\sum_{i=0}^{p+2} {m \choose i} \ell^{2^{O(\min(p,m))}}.
$
For fixed $m$, we obtain by letting $p = \ell$, 
an algorithm for computing all
the Betti numbers of $S$ whose complexity is $\ell^{2^{O(m)}}$.

The goal of this paper is to design an algorithm for computing the 
Betti numbers of semi-algebraic sets defined in 
terms of {\em partly quadratic} systems
of polynomials whose complexity interpolates between the doubly exponential
time bounds for the known algorithms in the general case, and the polynomial
complexity in case of semi-algebraic sets defined by few quadratic 
inequalities. 
Our algorithm is partly based on techniques  developed in
\cite{BP'R07jems}, where we prove a quantitative result 
{\em bounding} the Betti numbers of semi-algebraic
sets defined by partly quadratic systems of polynomials.
Before stating this result we introduce some notation that we are going to
fix for the rest of the paper.

\begin{notation}
We denote by
\begin{itemize}
\item
${\mathcal Q}\subset  \R[Y_1,\ldots,Y_\ell,X_1,\ldots,X_k]$,
a family of polynomials
with 
\[
\deg_{Y}(Q) \leq 2, 
\deg_{X}(Q) \leq d,  Q\in {\mathcal Q}, \#({\mathcal Q})=m,
\]
and by
\item
${\mathcal P} \subset \R[X_1,\ldots,X_k]$
a family of polynomials 
with
\[
\deg_{X}(P) \leq d, P \in {\mathcal P}, \#({\mathcal P})=s.
\]
\end{itemize}
\end{notation}

The following theorem is proved in \cite{BP'R07jems}.

\begin{theorem}
\label{the:main}
Let $S \subset \R^{\ell+k}$ 
be a $({\mathcal P} \cup {\mathcal Q})$-closed semi-algebraic set. Then
$$
\displaylines{
b(S) \leq 
\ell^2 (O(s+\ell+m)\ell d)^{k+2m}. 
}
$$
In particular, for $m \leq \ell$, we have
$
\displaystyle{
b(S) \leq \ell^2 (O(s+\ell)\ell d)^{k+2m}. 
}
$
\end{theorem}

The above theorem interpolates previously known bounds on the 
Betti numbers of general semi-algebraic sets (which are exponential
in the number of variables) 
\cite{OP,T,Milnor2,B99,GaV},
and bounds on Betti numbers of
semi-algebraic sets defined by quadratic inequalities (which are exponential
only in the number of inequalities and polynomial in the number of variables)
\cite{Bar97,Bas05-first-Kettner}.
Indeed we recover these extreme cases  by
by setting $\ell$ and $m$ 
(respectively, $s$, $d$ and $k$) to $O(1)$ in the above bound.

\subsection{Main Results}
The main result of this paper is algorithmic. We describe an
algorithm (Algorithm \ref{alg:general_betti} below) 
for computing all the Betti numbers of a closed 
semi-algebraic set defined by partly quadratic systems of polynomials.
The complexity of this algorithm interpolates the complexity of the
best known algorithms for computing the Betti numbers of general semi-algebraic
sets on one hand, and those described by quadratic 
inequalities on the other.

\begin{definition}[Complexity]
By complexity of an algorithm we will mean the number of arithmetic
operations (including comparisons) performed by the algorithm
in $\R$.
We refer the reader to \cite[Chapter 8]{BPRbook2} for a full discussion
about the various measures of complexity.
\end{definition}

We prove the following theorem.

\begin{theorem}
\label{the:algo-betti}
There exists an algorithm 
that takes as input the description of a
$({\mathcal P} \cup {\mathcal Q})$-closed semi-algebraic set
$S$ (following the same notation as in Theorem \ref{the:main})
and outputs its
Betti numbers $b_0(S),\ldots,b_{\ell+k-1}(S)$.
The complexity of this algorithm  is bounded by 
$(\ell s m d)^{2^{O(m+k)}}$.
\end{theorem}

The algorithm we describe is an adaptation of the 
algorithm in \cite{Bas05-top},
to the case where there are parameters, and the degrees with respect to these
parameters could be larger than two.
In addition, in this paper we also treat the case of general 
${\mathcal P}\cup {\mathcal Q}$-closed sets, not just basic closed ones
as was done in \cite{Bas05-top}.
We also provide more details and analyze the
complexity of the algorithm more carefully, in order to take into account the
dependence on the additional parameters.

\subsection{Significance from the point of view of 
computational complexity theory}

The problem of computing the Betti numbers of 
semi-algebraic sets in general is a PSPACE-hard problem.
We refer the reader to \cite{Bas05-top} and the references
contained therein, for a detailed discussion of these hardness results.
On the other hand, as shown in \cite{Bas05-top},
the problem of computing the Betti numbers of semi-algebraic
sets defined by a constant number of quadratic inequalities is solvable in
polynomial time. This result depends critically on the quadratic dependence
of the variables, as witnessed by the fact that
the problem of computing the Betti numbers of a 
real algebraic variety defined by
a single quartic equation is also PSPACE-hard. 
We show in this paper that the problem of computing the
Betti numbers of semi-algebraic sets defined by a constant number of
polynomial inequalities is solvable in polynomial time, even if we allow
a small (constant sized) subset of the variables to occur with degrees
larger than two in the polynomials defining the given set. 
Note that such a result is not obtainable directly from the results 
in \cite{Bas05-top} by the naive method of 
replacing the monomials having degrees larger than two by a larger set
of quadratic ones (introducing new variables and equations in the process).

The rest of the paper is organized as follows. In Section \ref{sec:proof}
we describe some mathematical results concerning the topology of sets
defined by quadratic inequalities. We often omit proofs if these 
appear elsewhere and just provide appropriate pointers to literature.
In Section \ref{sec:algo_betti} we describe our algorithm for computing
all the Betti numbers of semi-algebraic sets defined by partly quadratic
systems of polynomials and prove its correctness and complexity bounds,
thus proving Theorem \ref{the:algo-betti}.
 
\section{Topology of sets defined by partly quadratic systems of polynomials}
\label{sec:proof}
In this section we recall a construction described in \cite{BP'R07jems}
that will be important for the algorithm described later.
We parametrize a construction introduced by 
Agrachev in \cite{Agrachev} while studying the topology of sets defined by 
(purely) quadratic inequalities (that is without the parameters 
$X_1,\ldots,X_k$ in  our notation). 
However, 
we do not make any non-degeneracy assumptions
on our polynomials, and we also
avoid construction of Leray spectral sequences as done in 
\cite{Agrachev}.

We first need to fix some notation.

\subsection{Mathematical Preliminaries}
\subsubsection{Some Notation}
For all $a \in R$ we define 
\begin{eqnarray*}
\s(a) &=& 0  \mbox{ if } a = 0, \\
      &=& 1  \mbox{ if } a > 0, \\
      &=& -1 \mbox{ if } a < 0.
\end{eqnarray*}
Let ${\mathcal A}$ be a finite subset of  $\R[X_1,\ldots,X_k]$.
A  {\em sign condition}  on
${\mathcal A}$ is an element of $\{0,1,- 1\}^{\mathcal A}$.
The {\em realization of the sign condition}
$\sigma$, $\RR(\sigma,\R^k)$, is the basic semi-algebraic set
$$
        \{x\in \R^k\;\mid\; \bigwedge_{P\in{\mathcal A}} 
\s({P}(x))=\sigma(P) \}.
$$

A  {\em weak sign condition}  on
${\mathcal A}$ is an element of $\{\{0\},\{0,1\},\{0,-1\}\}^{\mathcal A}$.
The {\em realization of the weak sign condition}
$\rho$, $\RR(\rho,\R^k)$, is the basic semi-algebraic set
$$
        \{x\in \R^k\;\mid\; \bigwedge_{P\in{\mathcal A}} 
\s({P}(x)) \in \rho(P) \}.
$$

We
often abbreviate $\RR(\sigma,\R^k)$ by $\RR(\sigma)$, and we 
denote by ${\rm Sign}({\mathcal A})$ the set
of realizable sign conditions
${\rm Sign}({\mathcal A})=\{\sigma \in \{0,1,- 1\}^{\mathcal A} \;\mid\; \RR(\sigma) \neq \emptyset\}$.

More generally, for any ${\mathcal A} \subset \R[X_1,\ldots,X_k]$ and
a ${\mathcal A}$-formula $\Phi$, we 
denote by 
$\RR(\Phi,\R^k)$, or simply $\RR(\Phi)$,  the
semi-algebraic set defined by $\Phi$ in $\R^k$.

\subsubsection{Use of Infinitesimals}
Later in the paper,
we
extend the ground field $\R$ by infinitesimal
elements.
We denote by $\R\langle \zeta\rangle$  the real closed field of algebraic
Puiseux series in $\zeta$ with coefficients in $\R$ (see \cite{BPRbook2} for
more details). 
The sign of a Puiseux series in $\R\langle \zeta\rangle$
agrees with the sign of the coefficient
of the lowest degree term in
$\zeta$. 
This induces a unique order on $\R\langle \zeta\rangle$ which
makes $\zeta$ infinitesimal: $\zeta$ is positive and smaller than
any positive element of $\R$.
When $a \in \R\la \zeta \ra$ is bounded 
from above and below by some elements of $\R$,
$\lim_\zeta(a)$ is the constant term of $a$, obtained by
substituting 0 for $\zeta$ in $a$.
We denote by 
$\R\langle\zeta_1,\ldots,\zeta_n\rangle$ the
field  $\R\langle \zeta_1\rangle \cdots \langle\zeta_n\rangle$
and in this case 
$\zeta_1$ is positive and infinitesimally small compared to $1$, 
and for $1 \leq i \leq n-1$, 
$\zeta_{i+1}$ is positive and infinitesimally small 
compared to $\zeta_i$, which we abbreviate by writing
$0 < \zeta_n \ll \cdots \ll \zeta_1 \ll 1$.

Let $\R'$ be a real closed field containing $\R$.
Given a semi-algebraic set
$S$ in ${\R}^k$, the {\em extension}
of $S$ to $\R'$, denoted $\E(S,\R'),$ is
the semi-algebraic subset of ${ \R'}^k$ defined by the same
quantifier free formula that defines $S$.
The set $\E(S,\R')$ is well defined (i.e. it only depends on the set
$S$ and not on the quantifier free formula chosen to describe it).
This is an easy consequence of the transfer principle (see for instance
\cite{BPRbook2}).

\subsection{Homogeneous Case}
\label{subsec:homogeneous}

\begin{notation}
\label{not:AhWh}
We denote by
\begin{itemize}
\item  ${\mathcal Q}^h$
the  family of polynomials  
obtained by homogenizing ${\mathcal Q}$ 
with respect to the variables $Y$, i.e.
\[
{\mathcal Q}^h = \{Q^h \;\mid\; Q \in {\mathcal Q}\} 
\subset  \R[Y_0,\ldots,Y_\ell,X_1,\ldots,X_k],
\]
where $Q^h=Y_0^2 Q(Y_1/Y_0,\ldots,Y_\ell/Y_0,X_1,\ldots,X_k)$,
\item $\Phi$  a formula defining a ${\mathcal P}$-closed semi-algebraic set $V$,
\item $A^h$ the semi-algebraic set
\begin{equation}
\label{eqn:defofAh}
A^h = \bigcup_{Q \in {\mathcal Q}^h}
\{ (y,x) \;\mid\; |y|=1\; \wedge\; Q(y,x) \leq 0\; \wedge \; \Phi(x)\},
\end{equation}
\item $W^h$ the semi-algebraic set 
\begin{equation}
\label{eqn:defofWh}
W^h = \bigcap_{Q \in {\mathcal Q}^h}
\{ (y,x) \;\mid\; |y|=1\; \wedge\; Q(y,x) \leq 0\; \wedge \; \Phi(x)\}.
\end{equation}
\end{itemize}
\end{notation}

Let
\begin{equation}
\label{def:omega}
\Omega = \{\omega \in \R^{m} \mid  |\omega| = 1, \omega_i \leq 0, 1 \leq i \leq m\}.
\end{equation}

Let ${\mathcal Q}=\{Q_1,\ldots, Q_m \}$
and ${\mathcal Q}^h=\{Q_1^h,\ldots, Q_m^h \}$.
For $\omega \in \Omega$ we denote by 
$\la \omega ,{\mathcal Q}^h \ra \in \R[Y_0,\ldots,Y_\ell,X_1,\ldots,X_k]$ the polynomial
defined by 
\begin{equation}
\label{def:omegaq}
\la \omega , {\mathcal Q}^h \ra = \sum_{i=1}^{m} \omega_i Q_i^h.
\end{equation}

For $(\omega,x) \in \Omega \times V$, we 
denote by
$\la \omega , {\mathcal Q}^h \ra (\cdot,x)$ the quadratic form in $Y_0,\ldots,Y_\ell$ 
obtained from $\la \omega , {\mathcal Q}^h \ra$ by specializing $X_i = x_i, 1 \leq i \leq k$.

Let $B \subset \Omega \times \Sphere^{\ell} \times V$ 
be the semi-algebraic set defined by
\begin{equation}
\label{def:B}
B = \{ (\omega,y,x)\mid \omega \in \Omega, y\in \Sphere^{\ell}, x \in 
V,  \; \la \omega, {\mathcal Q}^h \ra (y,x) \geq 0\}.
\end{equation}

We denote by $\varphi_1: B \rightarrow F$ and 
$\varphi_2: B \rightarrow \Sphere^{\ell} \times V$ the two projection maps
(see diagram below).

\[
\begin{diagram}
\node{}
\node{B} \arrow{sw,t}{\varphi_{1}}\arrow[2]{s}\arrow{se,t}{\varphi_{2}} \\
\node{F = \Omega \times V} \arrow{se} \node{} \node{\Sphere^{\ell} \times V} \arrow{sw} \\
\node{}\node{V}
\end{diagram}
\]

The following key proposition was proved by Agrachev \cite{Agrachev}
in the unparametrized
situation, but as we see below it works in the parametrized case as well.

\begin{proposition}
\label{pro:homotopy2}
The semi-algebraic set $B$ is homotopy equivalent to
$A^h$.
\end{proposition}

\begin{proof}
See \cite{BP'R07jems}.
\end{proof}

We will use the following notation.
\begin{notation}
For a  quadratic form $Q \in \R[Y_0,\ldots,Y_\ell]$, 
we
denote by ${\rm index}(Q)$ the number of
negative eigenvalues of the symmetric matrix of the corresponding bilinear
form, i.e. of the matrix $M$ such that
$Q(y) = \langle M y, y \rangle$ for all $y \in \R^{\ell+1}$ 
(here $\langle\cdot,\cdot\rangle$ denotes the usual inner product). 
We also
denote by $\lambda_i(Q), 0 \leq i \leq \ell$ the eigenvalues of $Q$ 
in non-decreasing order, i.e.
\[ \lambda_0(Q) \leq \lambda_1(Q) \leq \cdots \leq \lambda_\ell(Q).
\]
\end{notation}

For $F=\Omega \times V$ as above we denote
\[
F_j = \{(\omega,x) \in F\;  
\mid \;  {\rm index}(\la \omega , {\mathcal Q}^h \ra (\cdot,x)) \leq j \}.
\]

It is clear that each 
$F_j$ is a closed semi-algebraic subset of 
$F$ and 
we get
a filtration of the space
$F$ 
given by
\[
F_0 \subset F_1 \subset \cdots \subset 
F_{\ell+1} = F.
\]

\begin{lemma}
\label{lem:sphere}
The fibre of the map $\varphi_1$ over a point 
$(\omega,x)\in F_{j}\setminus F_{j-1}$ 
has the homotopy type of a sphere of dimension $\ell-j$. 
\end{lemma}

\begin{proof}
See \cite{BP'R07jems}.
\end{proof}

For each 
$(\omega,x) \in F_j \setminus F_{j-1}$, let 
$L_j^+(\omega,x) \subset \R^{\ell+1}$ denote the sum of the
non-negative eigenspaces of 
$\la \omega , {\mathcal Q}^h \ra (\cdot,x)$.
Since  ${\rm index}(\la \omega, {\mathcal Q}^h \ra (\cdot,x)) = j$ stays invariant as
$(\omega,x)$ varies over $F_j
\setminus F_{j-1}$,
$L_j^+(\omega,x)$ varies continuously with $(\omega,x)$.

We denote by $C$ the semi-algebraic set defined by the following.
We first define for $0 \leq j \leq \ell+1$
\begin{equation}
\label{eqn:definition_of_C_j}
C_j = \{(\omega,y,x) \;\mid\; (\omega,x) \in 
      F_{j}\setminus F_{j-1}, 
y \in L_j^+(\omega,x), |y| = 1\},
\end{equation}
and finally we define
\begin{equation}
\label{eqn:definition_of_C}
C = \bigcup_{j=0}^{\ell+1} C_j.
\end{equation}

The following proposition proved in 
\cite{BP'R07jems}
relates the homotopy type of $B$ to that
of $C$.

\begin{proposition}
\label{pro:homotopy1}
The semi-algebraic set $C$ defined by 
\eqref{eqn:definition_of_C} is homotopy equivalent to $B$.
\end{proposition}

The following example which also appears in \cite{BP'R07jems} 
illustrates Proposition \ref{pro:homotopy1}.

\begin{example}
In this example $m=2,\ell = 3,k=0$, and 
${\mathcal Q}^h=\{Q_1^h,Q_2^h\}$ with
\begin{align*}
Q_1^h =& - Y_0^2 - Y_1^2 - Y_2^2, \\
Q_2^h =&   Y_0^2 + 2 Y_1^2  + 3 Y_2^2.
\end{align*}

The set $\Omega$ is the part of the unit circle in the third quadrant of the
plane, 
and  $F = \Omega$ in this case (since $k=0$).
In the following Figure \ref{fig:illus}, we display
the fibers of the map $\varphi_1^{-1}(\omega) \subset B$ for a sequence of 
values of $\omega$ starting from $(-1,0)$ and ending at 
$(0,-1)$. We also show the spheres,
$C \cap \varphi_1^{-1}(\omega)$, of dimensions $0,1$, and $2$, that these fibers
retract to. At $\omega = (-1,0)$, it is easy to verify that
${\rm index}(\la \omega, {\mathcal Q}^h \ra) = 3$, and the 
fiber $\varphi_1^{-1}(\omega) \subset B$
is empty. Starting from
$\omega = (-\cos(\arctan(1)),-\sin(\arctan(1)))$ we have 
${\rm index}(\la \omega ,{\mathcal Q}^h\ra) = 2$,
and the fiber 
$\varphi_1^{-1}(\omega)$ consists of the union of two spherical caps, 
homotopy equivalent to $\Sphere^0$.
Starting from
$\omega = (-\cos(\arctan(1/2)),-\sin(\arctan(1/2)))$ we have
${\rm index}(\la \omega,  {\mathcal Q}^h\ra) = 1$, and 
the fiber $\varphi_1^{-1}(\omega)$ is homotopy equivalent to $\Sphere^1$. Finally,
starting from
$\omega = (-\cos(\arctan(1/3)),-\sin(\arctan(1/3)))$,
${\rm index}(\la \omega , {\mathcal Q}^h\ra) = 0$, and 
the fiber $\varphi_1^{-1}(\omega)$ stays equal to to $\Sphere^2$.

\begin{figure}[hbt]
\scalebox{0.8}{
\begin{picture}(405,70)
\includegraphics{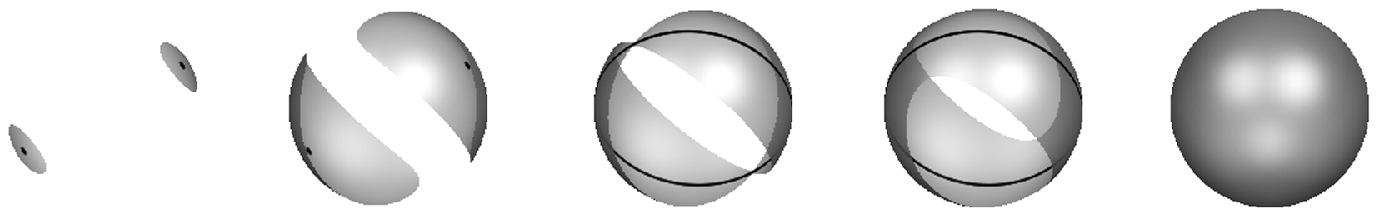}%
\end{picture}
}
\caption{Type change: $\emptyset\to \Sphere^0\to \Sphere^1\to \Sphere^2$. 
$\emptyset$ is not shown. }
\label{fig:illus}
\end{figure}
\end{example}

Let
$\Lambda \in \R[Z_1,\ldots,Z_m,X_1,\ldots,X_k,T]$ be the
polynomial defined by

\begin{eqnarray*}
\Lambda  &=& \det( T \cdot {{\rm Id}}_{\ell+1}-M_{Z \cdot {\mathcal Q}^h}),\\
   &=&  T^{\ell+1} + C_{\ell} T^\ell + \cdots + C_0,
\end{eqnarray*}
where $Z \cdot {\mathcal Q}^h = \sum_{i=1}^m Z_i Q_i^h$, and 
each $C_i \in \R[Z_1,\ldots,Z_m,X_1,\ldots,X_k]$.

Note that for $(\omega,x) \in \Omega \times\R^k$, the polynomial
$\Lambda(\omega,x,T)$, 
being the characteristic polynomial of a real symmetric
matrix, has all its roots real. 
It then follows from Descartes' rule of signs 
(see for instance \cite{BPRbook2}),
that for each $(\omega,x) \in \Omega \times \R^k$, 
${\rm index}(\la \omega ,{\mathcal Q}^h\ra (\cdot,x))$ is determined by the sign vector
\[
({\rm sign}(C_\ell(\omega,x)),\ldots,{\rm sign}(C_0(\omega,x))).
\] 
More precisely, the number of  
sign variations in the sequence
\[
\s(C_0(\omega,x)),\ldots,(-1)^i \s(C_i(\omega,x)),\ldots, (-1)^\ell \s(C_\ell(\omega,x)),+1
\]
is equal to ${\rm index}(\la \omega, {\mathcal Q}^h \ra (\cdot,x))$.

Hence, denoting 
\begin{equation}
\label{eqn:defincalC}
{\mathcal C} = 
\{C_0,\ldots,C_{\ell}\} \subset \R[Z_1,\ldots,Z_m,X_1,\ldots,X_k],
\end{equation}
we have 

\begin{lemma}
\label{lem:defofD_j}
$F_j$ is the intersection of 
$F$ with a ${\mathcal C}$-closed semi-algebraic set 
$D_j \subset \R^{m+k}$,
for each 
$0 \leq j \leq \ell+1$. \qed
\end{lemma}

\section{Computing the Betti numbers}
\label{sec:algo_betti}
We now consider the algorithmic problem of computing all the Betti
numbers of a semi-algebraic set defined by a partly quadratic 
system of polynomials. 
\subsection{Summary of the main idea}
The main idea behind the algorithm can be summarized as follows. 

By virtue of Proposition
\ref{pro:homotopy2}, in order to compute the Betti numbers of $A^h$, 
it suffices to construct a cell complex,
${\mathcal K}(B,V)$, whose associated space is homotopy equivalent
to the set $B$ defined by  \eqref{def:B}. 
In order to do so, we first compute a semi-algebraic triangulation,
$h: \Delta \rightarrow F$, 
such that 
as  $(\omega,x)$ varies over the image of any simplex $\sigma \in \Delta$,
the ${\rm index}(\la \omega , {\mathcal Q}^h \ra(\cdot,x))$ 
stays fixed,
and we have a continuous choice of an orthonormal basis,
\[
\{e_0(\sigma,\omega,x),\ldots,e_\ell(\sigma,\omega,x)\}
\]
consisting of eigenvectors of the symmetric matrix associated to the 
quadratic form $\la \omega ,{\mathcal Q}^h \ra(\cdot,x)$.

Moreover, if ${\rm index}(\la \omega ,{\mathcal Q}^h\ra (\cdot,x)) = j$ for 
$(\omega,x) \in h(\sigma)$, then $\varphi_1^{-1}(\omega,x)$ can be retracted to
$\Sphere^{\ell} \cap \spanof(e_j(\sigma,\omega,x),\ldots,
e_\ell(\sigma,\omega,x))$, and
the flag of subspaces defined by 
the orthonormal basis,
$\{e_0(\sigma,\omega,x),\ldots,e_\ell(\sigma,\omega,x)\}$, gives an efficient
regular cell decomposition of the sphere
$\Sphere^{\ell} \cap \spanof(e_j(\sigma,\omega,x),\ldots,
e_\ell(\sigma,\omega,x))$
into $2(\ell - j+1)$ cells, having two cells of each dimension from 
$0$ to $\ell - j$ (see Definition \ref{def:cell}). 

Now consider a pair of simplices, $\sigma,\tau \in \Delta$, with
$\sigma \prec \tau$.
The orthonormal basis 
$\{e_0(\tau,\omega,x),\ldots,e_\ell(\tau,\omega,x)\}$, 
defined for $(\omega,x) \in h(\tau)$ might not have a continuous extension
to $h(\sigma)$ on the boundary of $h(\tau)$. In particular, 
the cell decompositions of the fibers,
$\Sphere^{\ell} \cap \spanof(e_j(\sigma,\omega,x),\ldots,
e_\ell(\sigma,\omega,x))$,
over points in $(\omega,x) \in h(\sigma)$ might not be compatible with those 
over neighboring points in $h(\tau)$. In order to obtain a proper cell
complex we need to compute a common refinement of the cell decomposition
of the sphere over each point in $(\omega,x) \in h(\sigma)$ induced by
the basis $\{e_0(\sigma,\omega,x),\ldots,e_\ell(\sigma,\omega,x)\}$,
and the one obtained as a limit of those over certain points in $h(\tau)$ 
converging to $(\omega,x)$. We need to further subdivide 
$h(\sigma)$ to ensure that over each cell of this sub-division the 
combinatorial
type of the above refinements stays the same. 
Since a simplex $\sigma \in \Delta$ can be incident on many other
simplices of $\Delta$, we might in the above procedure 
need to simultaneously refine cell 
decompositions of the sphere coming from many different simplices.
In order to ensure (for complexity reasons) that we do not have to 
simultaneously refine cell decompositions coming from too many simplices, 
we thicken the simplices infinitesimally and as a result only need to
refine at most $m+k$ cell decompositions at a time.      

Before describing the construction of ${\mathcal K}(B,V)$ in more detail,
we need some preliminaries on triangulations.

\subsection{Triangulations}
We first need to recall a fact from semi-algebraic geometry about triangulations
of semi-algebraic sets, and then we define the notion of an Index Invariant Triangulation
and give an algorithm 
for
computing it.

\subsubsection{Triangulations of semi-algebraic sets}
A triangulation
of a closed and bound\-ed semi-algebraic set $S$ is a simplicial complex 
$\Delta$ together with a
semi-algebraic homeomorphism from $\vert \Delta\vert$ to $S$. 
We always assume that the simplices in $\Delta$ are open.
Given such a triangulation we will often identify the simplices in
$\Delta$ with their images in $S$ under the given homeomorphism, and
will  refer to the triangulation by $\Delta$.

Given a triangulation $\Delta$, the cohomology groups 
$\HH^i(S)$ are isomorphic to the simplicial cohomology groups 
$\HH^i(\Delta)$ of the simplicial complex  $\Delta$ and
are in fact independent of the triangulation $\Delta$ 
(this fact is classical over $\re$;  see for instance \cite{BPRbook2} 
for a self-contained proof in the category of semi-algebraic sets).
 
We call a triangulation $h_1: |\Delta_1| \rightarrow S$
of a semi-algebraic set $S$, to be a {\em refinement}
of a triangulation
$h_2: |\Delta_2| \rightarrow S$ if
for every simplex $\sigma_1 \in \Delta_1$, there exists a simplex
$\sigma_2 \in \Delta_2$ such that $h_1(\sigma_1) \subset h_2(\sigma_2).$

Let $S_1 \subset S_2$ be two compact semi-algebraic subsets of
$\R^k$. We say that a semi-algebraic 
triangulation $h: |\Delta| \rightarrow S_2$ of $S_2$, respects
$S_1$ if for every simplex $\sigma  \in \Delta$,
$h(\sigma) \cap S_1 = h(\sigma)$ or $\emptyset$.
In this case, $h^{-1}(S_1)$ is identified with a sub-complex
of $\Delta$ and
$h|_{h^{-1}(S_1)} :h^{-1}(S_1) \rightarrow S_1$ is a semi-algebraic 
triangulation  of $S_1$. We will refer to this sub-complex  by 
$\Delta|_{S_1}$.

We will need the following theorem which can be deduced 
from Section 9.2 in  \cite{BCR} (see also \cite{BPRbook2}).

\begin{theorem}
\label{the:triangulation}
Let $S_1 \subset S_2 \subset\R^k$ 
be closed and bounded semi-algebraic sets, and
let $h_i: \Delta_i \rightarrow S_i, i = 1,2$ be 
semi-algebraic triangulations of
$S_1,S_2$. Then there exists a semi-algebraic triangulation
$h: \Delta \rightarrow S_2$ of $S_2$, such that
$\Delta$ respects $S_1$,
$\Delta$ is a refinement of $\Delta_2$, and
$\Delta|_{S_1}$ is a refinement of $\Delta_1$.

Moreover, there exists an
algorithm which computes such a triangulation 
with complexity bound
$(sd)^{O(1)^{k}}$, where $s$ is the number of polynomials used in the
definition of $S_1$ and $S_2$, and $d$ is a bound on their degrees. 
\end{theorem}

\subsubsection{Parametrized eigenvector basis}
\label{subsec:parametrizedeigenvectors}
Let $M(\omega,x)$ be the symmmetric matrix associated to the quadratic form 
$\la \omega , {\mathcal Q}^h \ra(\cdot,x)$ defined by
\eqref{def:omegaq}. 
When $M(\omega,x)$  has simple eigenvalues for all possible choice of 
$\omega,x$
in some domain, 
there is a finite choice of 
orthonormal bases consisting of eigenvectors of $M(\omega,x)$.
However, when $M(\omega,x)$ has multiple eigenvalues, 
the number of choices of 
orthonormal basis of eigenvectors is infinite.
In order to avoid the problem caused by the latter situation 
we are going to use an infinitesimal deformation as follows.

Let $0 < \eps \ll 1$ be an infinitesimal and 
\begin{equation}
\label{eqn:defofM_eps}
M_\eps(\omega,x)=(1-\eps)M(\omega,x)+\eps\diag(0,1,2,\dots,\ell).
\end{equation}
Note that for every $(\omega,x)\in \Omega \times \R^{k}$
the eigenvalues of $M_\eps(\omega,x)$ in $\R\la \eps \ra$ are distinct and nonzero.
Indeed,
replace $\eps$ by $t$ in the definition of 
$M_\eps(\omega,x)$ and obtain $M_t(\omega,x)$. 
Observe that the statement is true if $t=1$, since
the matrix $M_1(\omega,x)$  has distinct eigenvalues. Thus, the  set of $t$'s 
in the algebraically closed field $\R[i]$
for which   $M_t(\omega,x)$ 
has $\ell+1$ distinct eigenvalues is non-empty,
constructible and contains a open subset, since the condition of having distinct
eigenvalues is a stable condition. Thus, there exists $\eps_0 > 0$, such that
for all $t\in (0,\eps_0)$, 
$M_t(\omega,x)$ has $\ell+1$ distinct eigenvalues, and hence
it is also the case for the infinitesimal $\eps$. 

Denote by 
$\Lambda(M_\eps(\omega,x,T))=\det(T \cdot \Id_{\ell+1}-M_\eps(\omega,x))$ 
the characteristic polynomial of  $M_\eps(\omega,x)$.
Let $\mathcal A \subset \R[Z_1,\ldots, Z_m,X_1,\ldots,X_k]$ be a set of 
polynomials containing $\mathcal C$ 
(see \eqref{eqn:defincalC})
 and such that
for every sign condition $\rho \in \{0,1,-1\}^{\mathcal A}$ and 
every $(\omega,x)\in \RR(\rho,\Omega \times \R^k)$, the Thom encodings of 
the roots of $\Lambda(M_\eps(\omega,x),T)$
stay fixed, as well as the list of the non-singular minors of size 
$\ell$ in $M_\eps(\omega,x,T)$
at each root of $\Lambda(M_\eps(\omega,x),T)$.

Then choosing a non-vanishing minor and using Cramer's rule, 
we find 
$(\ell+1)^2$
rational functions in the variables 
$u,\omega,x,T$ 
which give for every 
$(u,\omega,x)\in \RR(\rho,\Omega \times \R^{k+1})$ the coordinates of an 
eigenvector $v_\eps(u,\omega,x,t_\eps)$ associated to the eigenvalue $t_\eps$
(where $u$ denotes the co-ordinate left out in the non-singular 
$\ell \times \ell$ minor chosen for this eigenvalue in the application of
Cramer's rule).
We denote by $e_\eps(\omega,x,t_\eps)$  the unit eigenvector 
$v_\eps(1,\omega,x,t_\eps)/\Vert v_\eps(1,\omega,x,t_\eps) \Vert$
when $t_\eps$ is an eigenvalue.

If the eigenvalues $\lambda_{\eps,0} <\ldots <\lambda_{\eps,\ell}$ are in 
increasing order, we define  
\[
e_{\eps,i}(\omega,x)=e_\eps(\omega,x,\lambda_{\eps,i}).
\]
Note that for every 
$(\omega,x)\in \Omega \times \R^{k}$
\[
(\lim_\eps(e_{\eps,0}(\omega,x)), \ldots,\lim_\eps(e_{\eps,\ell}(\omega,x))) 
\]
is an orthonormal basis consisting of eigenvectors of $M(\omega,x)$.

\subsubsection{Index Invariant Triangulations}
We now define a certain special kind of 
semi-algebraic triangulation of $F$
that will play an important role in our algorithm.

\begin{definition}(Index Invariant Triangulation)
\label{def:iit}
An {\em index invariant triangulation} of $F$ is a triangulation
$$
h: \Delta  \rightarrow F
$$ of  $F$, 
which respects all the realization of the 
weak sign conditions on ${\mathcal P}$ and ${\mathcal A}$
(see 
definition in \ref{subsec:parametrizedeigenvectors}).
As a consequence, $h$ respects the subsets $F_{I}$ for every
$I \subset {\mathcal Q}$. Moreover,
${\rm index}(\la \omega , {\mathcal Q}^h \ra(\cdot,x))$, 
stays invariant as $(\omega,x)$ varies over $h(\sigma)$, and
the maps $e_{\eps,0}(\sigma),\ldots,e_{\eps,\ell}$ sending $(\omega,x)\in h(\sigma)$ to
the orthonormal basis
 $e_{\eps,0}(\omega,x),\ldots, e_{\eps,\ell}(\omega,x)$,
are uniformly defined.
Note also that for every $(\omega,x)\in h(\sigma)$,
\[
\{e_j(\sigma,\omega,x)),\ldots,e_{\ell}(\sigma,\omega,x))=\{\lim_\eps(e_{\eps,j}(\omega,x)),\ldots,\lim_\eps(e_{\eps,\ell}(\omega,x))\}
\] 
is a basis for 
the linear subspace $L^+(\omega,x) \subset \R^{\ell+1}$, 
(which is the orthogonal complement to the sum of the
eigenspaces corresponding to the
first $j$ eigenvalues of $\la \omega , {\mathcal Q}^h \ra (\cdot,x)$).
\end{definition}

We now describe  an  algorithm for computing index invariant
triangulations.

\begin{algorithm}[Index Invariant Triangulation]
\label{alg:triangulation}
\item[]
\item[{\sc Input}]
\item[]
\begin{itemize}
\item 
A family of polynomials,
$
{\mathcal Q}^h = \{Q_1^h,\ldots,Q_m^h\} 
\subset \R[Y_0,\ldots,Y_\ell,X_1,\ldots,X_k],
$
where each $Q_i^h$ is homogeneous of degree $2$ in the 
variables $Y_0,\ldots,Y_\ell$,
and of degree at most $d$ in $X_1,\ldots,X_k$,
\item another family of polynomials,
$
{\mathcal P} 
\subset \R[X_1,\ldots,X_k],
$
with $\deg(P) \leq d,  P \in {\mathcal P}, \#({\mathcal P})=s$,
\item
a ${\mathcal P}$-closed formula $\Phi$ defining a bounded
${\mathcal P}$-closed semi-algebraic set $V \subset \R^k$.
\end{itemize}

\item [{\sc Output}]: an index invariant triangulation,
$$
h: \Delta  \rightarrow F
$$ of  $F$
and for each simplex $\sigma$ of $\Delta$, the rational functions $e_{\eps,0}(\sigma),\ldots, e_{\eps,\ell}(\sigma).$
\item [{\sc Procedure}]
\item[]
\item[Step 1.]
Let $\eps > 0$ be an infinitesimal and 
let $Z= (Z_1,\ldots,Z_m)$.
Let $M_\eps$ be 
the symmetric matrix corresponding to the quadratic form
(in $Y_0,\ldots,Y_\ell$) defined  by  
\[
M_\eps(X,Z) = 
(1-\eps)(Z_{1}Q_{1}^h +\cdots+ Z_{m}Q_{m}^h) + \eps \bar{Q},
\]
where $\bar{Q} = \sum_{i=0}^{\ell} i Y_i^2$.
Compute the polynomials
\begin{align}
\Lambda(Z,X,T) =& \det  (T\cdot {{\rm Id}}_{\ell+1}-M_\eps)= 
T^{\ell+1} + C_{\ell}T^{\ell} + \cdots+ 
C_0.
\end{align}

\item[Step 2.]
Using Algorithm 11.19 in \cite{BPRbook2} (Restricted Elimination),
compute a family of polynomials 
${\mathcal A}' \subset \R[\eps][Z_1,\ldots,Z_m,X_1,\ldots,X_k]$ such that 
for each $\rho \in {\rm Sign}({\mathcal A}')$,
and $(\omega,x) \in \RR(\rho,\Omega \times \R^{k}) \cap F$ 
the Thom encodings of the roots of $\Lambda(\omega,x,T)$ in 
$\R\la\eps\ra$ and  the number of non-negative roots of $\Lambda(\omega,x,T)$ in $\R\la\eps\ra$ stay fixed, as well as the
list of the non singular minors of size $\ell$ in $M_\eps(\omega,x,T)$
at each root of $\Lambda(M_\eps(\omega,x),T)$.
Let ${\mathcal A} \subset \R[Z_1,\ldots,Z_m,X_1,\ldots,X_k]$ be 
the set of all coefficients of the polynomials in ${\mathcal A}'$,
when each of them is written as a polynomial in $\eps$.

\item[Step 3.]
Using the algorithm implicit in Theorem \ref{the:triangulation} 
(Triangulation),
compute a semi-algebraic triangulation,
$$
h: \Delta \rightarrow F,
$$
respecting all the realizations of the weak sign conditions
on  ${\mathcal A}\cup {\mathcal P}$.

\item[Step 4.]
For each simplex $\sigma$ of $\Delta$, output the maps
$e_{\eps,0}(\sigma),\ldots,e_{\eps,\ell}(\sigma)$.

\end{algorithm}

\vspace{.1in}
\noindent
{\sc Complexity Analysis:}
The complexity of the algorithm is dominated by the complexity
of Step 3, which is $(s \ell m d)^{2^{O(m+k)}}$.
\qedsymbol

\vspace{.1in}
\noindent
{\sc Proof of Correctness:}
It follows 
from the fact that the triangulation respects all weak sign conditions on
${\mathcal A}$ that 
${\rm index}(\la \omega , {\mathcal Q}^h \ra (\cdot,x))$ 
is constant for $(\omega,x) \in h(\sigma)$ 
for any simplex $\sigma$ of $\Delta$,

Since $e_{\eps,0}(\sigma,\omega,x),$ $\ldots$, $e_{\eps,\ell}(\sigma,\omega,x)$ 
are orthonormal, so are $e_0(\sigma,\omega,x),$ $\ldots$, \\
$e_\ell(\sigma,\omega,x)$ 
for every $(\omega,x) \in h(\sigma)$. 
Moreover, letting $j = {\rm index}(\la \omega , {\mathcal Q}^h \ra (\cdot,x))$ for 
$(\omega,x) \in h(\sigma)$, we have that
$e_{\eps,j}(\sigma,\omega,x),\ldots,e_{\eps,\ell}(\sigma,\omega,x)$ 
span the sum of the non-negative eigenspaces
of $M_\eps(\omega,x)$, 
their images under the $\lim_\eps$ map will span the
sum of the non-negative eigenspaces of  $M(\omega,x)$.  \qed

\subsection{Computing Betti numbers in the homegeneous union case}
\label{subsec:cellcomplex}
Now that we obtained an Index Invariant Triangulation $\Delta$, 
our next goal is to construct a cell complex ${\mathcal K}(B,V)$ homotopy equivalent to
$B$ (see Notation \ref{def:B}) that will be used to compute the Betti numbers of $A^h$ (see Notation \ref{not:AhWh}). The cell complex ${\mathcal K}(B,V)$ is obtained by glueing together certain regular 
cell complexes,  ${\mathcal K}(\sigma)$,
where $\sigma \in \Delta$.

       \begin{figure}[hbt]
         \centerline{
           \scalebox{0.5}{
             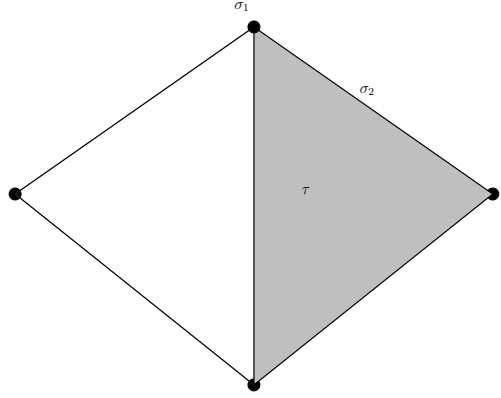
             }
           }
         \caption{The complex $\Delta$.}
         \label{fig-eg1}
       \end{figure}

       \begin{figure}[hbt]
         \centerline{
           \scalebox{0.5}{
             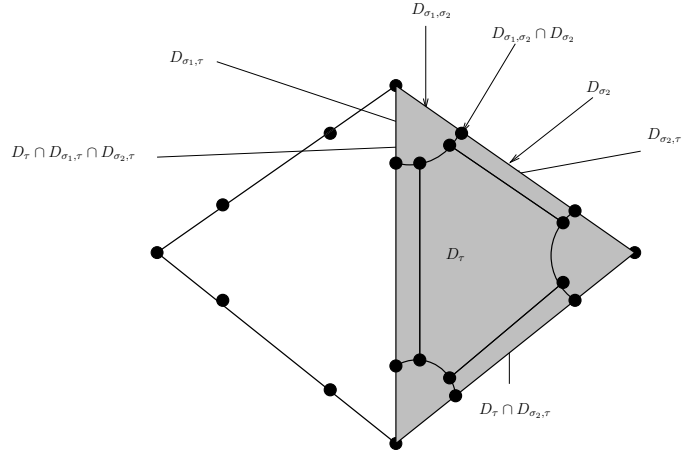
             }
           }
         \caption{The corresponding complex 
${\mathcal C}(\Delta)$.}
         \label{fig-eg2}
       \end{figure}

\subsubsection{Definition of ${\mathcal C}(\Delta)$}
Let $1 \gg  \eps_0 \gg \eps_1 \gg \cdots \gg \eps_{m+k} > 0$ be 
infinitesimals. 
For $\tau \in \Delta$, 
we denote by $D_{\tau}$ the subset of $\bar\tau$ defined by
\[
D_{\tau} = \{v \in \bar\tau \;\mid\; 
\dist(v,\theta) \geq \eps_{\dim(\theta)} \mbox{ for all }
\theta \prec \sigma \},
\]
where $\dist$ refers to the ordinary Euclidean distance.
Now, let $\sigma \prec \tau$ be two simplices of $\Delta$.
We denote by 
$D_{\sigma,\tau}$ the subset of $\bar\tau$ defined by
\[
D_{\sigma,\tau} = \{v \in \bar\tau \;\mid\; \dist(v,\sigma) \leq 
\eps_{\dim(\sigma)},
\mbox{ and }\dist(v,\theta) \geq \eps_{\dim(\theta)} \mbox{ for all }
\theta \prec \sigma \}.
\]

Note that 
$$
|\Delta| = 
\bigcup_{\sigma \in \Delta} D_{\sigma} \cup
\bigcup_{\sigma,\tau \in \Delta,\sigma \prec \tau} D_{\sigma,\tau}.
$$

Also, observe that the various $D_\tau$'s and $D_{\sigma,\tau}$'s 
are all homeomorphic to closed balls,
and moreover all non-empty intersections between them also have the same property.
Thus, the union of the $D_{\tau}$'s and $D_{\sigma,\tau}$'s together with the
non-empty intersections between them form  a regular cell complex,
${\mathcal C}(\Delta)$, whose underlying
topological space is $|\Delta|$ 
(see Figures \ref{fig-eg1} and \ref{fig-eg2}).

\subsubsection{Definition of ${\mathcal K}(\sigma)$ and ${\mathcal K}(\sigma,\tau)$ where $\sigma,\tau$ are simplices of $\Delta$ }
We now associate to each 
$D_{\sigma}$  (respectively,  $D_{\sigma,\tau}$)
a regular cell complex, ${\mathcal K}(\sigma)$, (respectively,
${\mathcal K}(\sigma,\tau)$)
homotopy equivalent to 
$\varphi_{1}^{-1}(h(D_\sigma))$
(respectively,
$
\displaystyle{
\varphi_{1}^{-1}(h(D_{\sigma,\tau})).
}
$

For each $\sigma \in \Delta$, 
and $(\omega,x) \in h(\sigma)$,
the orthonormal basis 
\[
\{e_{0}(\sigma,\omega,x)),\ldots,e_{\ell}(\sigma,\omega,x)\}
\] 
determines a complete flag of subspaces, 
${\mathcal F}(\sigma,\omega,x)$, consisting of
\begin{align*}
F^0(\sigma,\omega,x) = &0, \\
F^1(\sigma,\omega,x) = &\spanof(e_\ell(\sigma,\omega,x)),\\
F^2(\sigma,\omega,x) = &
\spanof(e_\ell(\sigma,\omega,x),e_{\ell-1}(\sigma,\omega,x)), \\
\vdots & \\
F^{\ell+1}(\sigma,\omega,x) =& \R^{\ell+1}.
\end{align*}

\begin{definition}
\label{def:cell}
For $0 \leq j \leq \ell$, let $c_{j}^+(\sigma,\omega,x)$ 
(respectively, $c_{j}^-(\sigma,\omega,x)$)
denote the $(\ell-j)$-dimensional cell consisting of the intersection of the
$F^{\ell-j+1}(\sigma,\omega,x)$
with the unit hemisphere in $\R^{\ell+1}$ 
defined by 
\begin{align*}
\{y \in \Sphere^\ell
\;\mid\;&
 \langle y,e_j(\sigma,\omega,x)\rangle \geq 0\}\\
\text{(respectively, } \quad
\{y \in \Sphere^\ell \;\mid\;& \langle y,e_j(\sigma,\omega,x)
\rangle \leq 0\}\quad
).
\end{align*}
\end{definition}

The regular cell complex ${\mathcal K}(\sigma)$ 
(as well as $\mathcal{K}(\sigma,\tau)$)
is defined as follows.

For each $v \in |\Delta|$ and $\sigma \in \Delta$, let 
$v(\sigma) \in |\sigma|$ denote the point of $|\sigma|$
closest to $v$.

The cells of ${\mathcal K}(\sigma)$ are
\[
\{(y,\omega,x) \mid y \in c_j^{\pm}(\sigma,\omega,x), (\omega,x) 
\in h(c)\},
\]
where ${\rm index}(\la \omega , {\mathcal Q}^h \ra(\cdot,x)) \leq j \leq \ell$,
and 
$c \in {\mathcal C}(\Delta)$
is either $D_\sigma$ itself, or a cell
contained in the boundary  of $D_\sigma$.

Similarly, the cells of ${\mathcal K}(\sigma,\tau)$ are
\[
\{(y,\omega,x) \mid y \in c_j^{\pm}(\sigma,h(v(\sigma))), v = h^{-1}(\omega,x) 
\in c\},
\]
where 
${\rm index}(\la \omega , {\mathcal Q}^h \ra(\cdot,x)) \leq j \leq \ell$,
$c \in {\mathcal C}(\Delta)$
is either $D_{\sigma,\tau}$ itself,  or a cell
contained in the boundary  of $D_{\sigma,\tau}$.

\subsubsection{Definition of ${\mathcal K}(D)$, 
where $D$ is a cell of ${\mathcal C}(\Delta)$}

Our next step is to obtain cellular subdivisions
of each non-empty intersection amongst the
spaces associated to the complexes constructed above, and thus obtain
a regular cell complex,
${\mathcal K}(B,V)$, whose associated space,
$|{\mathcal K}(B,V)|$, will be shown to be 
homotopy equivalent to $B$ (Proposition \ref{pro:iso2} below).

First notice that $|{\mathcal K}(\sigma',\tau')|$ (respectively, 
$|{\mathcal K}(\sigma)|$) has a non-empty intersection with 
$|{\mathcal K}(\sigma,\tau)|$ only if $D_{\sigma',\tau'}$ (respectively,
$D_{\sigma'}$) intersects $D_{\sigma,\tau}$. 

Let $D$ be some non-empty intersection amongst the 
$D_{\sigma}$'s and $D_{\sigma,\tau}$'s,
that is $D$ is a cell of ${\mathcal C}(\Delta)$.
Then $D \subset |\tau|$ for a unique simplex $\tau \in \Delta$, 
and 
$$
\displaylines{
D =  D_{\sigma_1,\tau} \cap \cdots \cap D_{\sigma_p,\tau} 
\cap D_\tau,
}
$$
with $\sigma_1 \prec \sigma_2 \prec \cdots \prec \sigma_p \prec \sigma_{p+1} =
\tau$
and $p \leq m+k$. 

For each $i, 1 \leq i \leq p+1$,
let
$ \{ f_0(\sigma_i,v),\ldots, f_{\ell}(\sigma_i,v)\}$
denote a orthonormal basis of $\R^{\ell+1}$ where
\[
f_j(\sigma_i,v) = 
\lim_{t \rightarrow 0}
e_j(\sigma_i, h(t v(\sigma_i) + (1-t) v(\sigma_1))), 0 \leq j \leq \ell,
\] 
and let
${\mathcal F}(\sigma_i,v)$
denote the corresponding flag,
consisting of
\begin{align*}
F^0(\sigma_i,v) = &0, \\
F^1(\sigma_i,v) = &\spanof(f_\ell(\sigma_i,v)),\\
F^2(\sigma_i,v) = &
\spanof(f_\ell(\sigma_i,v),f_{\ell-1}(\sigma_i,v)), \\
\vdots & \\
F^{\ell+1}(\sigma_i,v) =& \R^{\ell+1}.
\end{align*}
 
We thus have $p+1$ different flags, 
\[
{\mathcal F}(\sigma_1,v),
\ldots, {\mathcal F}(\sigma_{p+1},v), 
\]
and these give rise to $p+1$ different regular cell decompositions of 
$\Sphere^\ell$.

       \begin{figure}[h]
         \centerline{
           \scalebox{0.5}{
             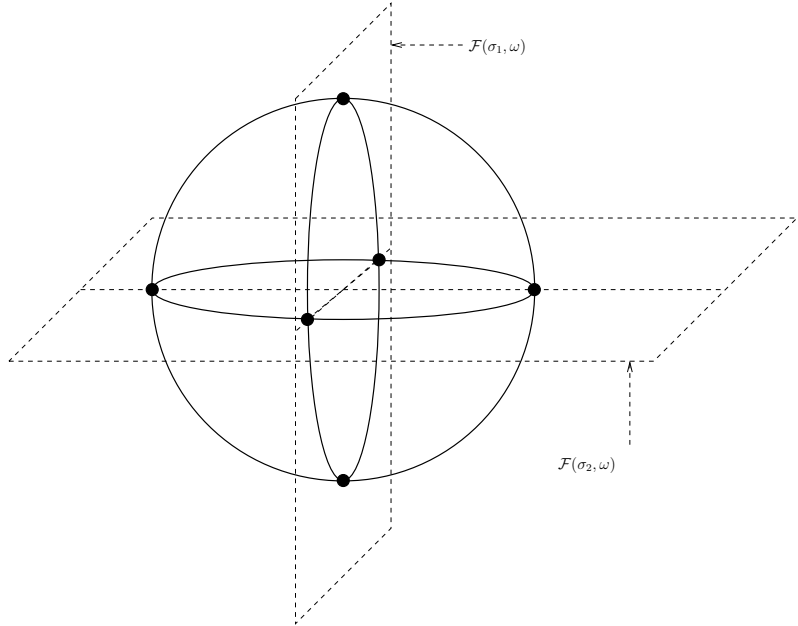
             }
           }
\caption{The cell complex ${\mathcal K}'(D,v)$.}  
         \label{fig-eg5}
       \end{figure}

There is a unique smallest regular cell complex, 
${\mathcal K}'(D,v)$,  
that refines all these cell decompositions,
whose cells are the following.
Let $L \subset \R^{\ell+1}$ be  any $j$-dimensional linear subspace,
$0 \leq j \leq \ell+1$, which is an intersection
of linear subspaces $L_1,\ldots,L_{p+1}$, where 
$L_i \in {\mathcal F}(\sigma_i,v), 1 \leq i \leq p+1 \leq m+k+1$. 
The elements of the flags,
${\mathcal F}(\sigma_1,v),
\ldots, {\mathcal F}(\sigma_{p+1},v)$ of dimensions $j+1$,
partition $L$ into polyhedral cones of various dimensions. The 
intersections of these cones with $\Sphere^{\ell}$, 
over all such subspaces $L \subset \R^{\ell+1}$, are the cells of 
${\mathcal K}'(D,v)$.
Figure \ref{fig-eg5} illustrates the refinement described above in case
of two flags in $\R^3$.
We 
denote by ${\mathcal K}(D,v)$ the sub-complex of 
${\mathcal K}'(D,v)$ consisting of only those cells included in
$L^+(\sigma_1,h(v(\sigma_1))) \cap \Sphere^\ell$.

We now triangulate $h(D)$ using 
the algorithm implicit in Theorem  \ref{the:triangulation} (Triangulation),
so that the combinatorial type of the arrangement of flags, 
\[
{\mathcal F}(\sigma_1,v),
\ldots, {\mathcal F}(\sigma_{p+1},v)
\] 
and hence 
the combinatorial type of the 
cell decomposition ${\mathcal K}'(D,v)$,  
stays invariant over the image,
$h_D(\theta)$,  of  each simplex, $\theta$, of this triangulation.
\hide{
Notice that the combinatorial type of the cell 
decomposition ${\mathcal K}'(D,v)$,  is determined by the signs of the inner 
products,
$\la f_j(\sigma_i,v), f_{j'}(\sigma_{i'},v) \ra$ where
$0 \leq j,j' \leq \ell, 1 \leq i,i' \leq p+1$.
}

Note that in case the eigenvalues of $M(h(v))$ are all distinct,
we have that
${\mathcal F}(\sigma_1,v)= \cdots ={\mathcal F}(\sigma_{p+1},v)$
since the vectors 
\[
f_0(\sigma_i,v),\ldots,f_\ell(\sigma_i,v) 
\]
is an orthonromal basis of eigen-vectors with each
$f_j(\sigma_i,v)$ uniquely defined upto sign.
Thus the cell decompositions of $\Sphere^\ell$ induced by the
flags 
${\mathcal F}(\sigma_1,v)= \cdots ={\mathcal F}(\sigma_{p+1},v)$
are identical to each other and hence also to ${\mathcal K}'(D,v)$.

However, if the eigenvalues of $M(h(v))$ are not all distinct
then the refinement  ${\mathcal K}'(D,v)$ can be
non-trivial.
For example suppose we have that
$\lambda_{\alpha}(h(v)) = \cdots = \lambda_{\beta}(h(v)), 0 \leq \alpha < \beta
\leq \ell$.
Then in general the sub-flags consisting of
subspaces  $F_{\ell+1-\beta}(\sigma_i,v) \subset \cdots
\subset  F_{\ell+1-\alpha}(\sigma_i,v)$ and
$F_{\ell+1-\beta}(\sigma_j,v) \subset \cdots\subset 
F_{\ell+1-\alpha}(\sigma_j,v)$
will in general not coincide for $i \neq j$.

In this case the combinatorial type of the refinement
${\mathcal K}'(D,v)$ is determined by the 
dimensions of the intersections 
amongst the subspaces 
\[
F_{\ell+1-\beta}(\sigma_i,v),
\ldots, F_{\ell+1-\alpha}(\sigma_i,v), 1 \leq i \leq p+1.
\]
The dimensions of these intersections are 
determined by the minimal linear dependencies amongst the vectors
$f_{i}(\sigma_{j},v)$, $\alpha \leq i \leq \beta, 1 \leq j \leq p+1$,
and these are in turn 
determined by the ranks of matrices with at most $\ell+1$ rows 
of the following form.
The rows of the matrix consists of at most $p+1$ blocks,
with the $j$-th block of the shape
$f_{\alpha}(\sigma_{j},v),\ldots, f_{\alpha_j}(\sigma_{j},v)$,
where $\alpha \leq \alpha_j \leq \beta$.
(Note that every row of the above matrix consists of rational functions 
evaluated at a single root $\lambda_\alpha(h(v))$ of $\Lambda(M(h(v)),T)$,
and this root is common to all the rows. This fact is important since
in order to perform algebraic computations on the entries of the matrix
we need  to eliminate just one variable corresponding to this single root.)

Introducing an infinitesimal $\delta$ such that $1 \gg \delta \gg \eps >0$, 
we note that for each $0 \leq j \leq \ell$,
\begin{align*}
f_j(\sigma_i,h^{-1}(\omega,x)) &= \lim_\delta
e_{\eps,j}(h(\delta v(\sigma_i) + (1-\delta) v(\sigma_1))) \\
&= 
\lim_{t \rightarrow 0}
e_j(\sigma_i, h(t v(\sigma_i) + (1-t) v(\sigma_1))).
\end{align*}

We consider all matrices with at most
$\ell+1$ rows consisting of blocks of the
shape, 
$f_{\alpha}(\sigma_{j},v),\ldots, f_{\alpha_j}(\sigma_{j},v)$, with
$0 \leq \alpha \leq \alpha_j \leq \beta \leq \ell, 0 \leq j \leq m+k$, and
$\lambda_{\alpha}(h(v)) = \cdots = \lambda_{\beta}(h(v))$.
The number of such matrices is clearly bounded by $\ell^{O(m+k)}$.

Using the uniform formula defining $e_{\eps,j}(\sigma_i)$ and 
Proposition 14.7 of \cite{BPRbook2}, and Algorithm 8.16 
in \cite{BPRbook2} (for computing determinants over an arbitrary domain),
we compute a family of polynomials in $\R[Z_1,\ldots,Z_m,X_1,\ldots,X_k]$
such that over each sign consition of this family the rank of the
given matrix stays fixed.

\hide{
Using the uniform formula defining $e_{\eps,j}(\sigma_i)$ and Proposition 14.7 of \cite{BPRbook2},
the vanishing or non-vanishing of the 

inner products
\[
\la f_j(\sigma_i,h^{-1}(\omega,x)), f_{j'}(\sigma_{i'},h^{-1}(\omega,x)) \ra, 
0 \leq j,j' \leq \ell, 1 \leq i,i' \leq p+1.
\]

\[
{\mathcal A}_{D}\subset \R[Z_1,\ldots,Z_m,X_1,\ldots,X_k]
\]
obtained by taking the coefficients of the 
inner products
\[
\la e_{\eps,j}(h(\delta v(\sigma_i) + (1-\delta) v(\sigma_1))), e_{\eps,j'}(h(\delta v(\sigma_i') + (1-\delta) v(\sigma_1)))\ra
\]
}

Let
\[
{\mathcal A}_{D}\subset \R[Z_1,\ldots,Z_m,X_1,\ldots,X_k]
\]
be the union of all these sets of polynomials.

The combinatorial type of the cell 
decomposition ${\mathcal K}'(D,v)$
will stay invariant as $(\omega,x)$ varies over each connected component of 
any realizable sign condition on 
${\mathcal A}_{D}\subset \R[Z_1,\ldots,Z_m,X_1,\ldots,X_k]$.

Given the degree bounds on the rational functions defining 
$\{e_{\eps,0}(\sigma),\ldots,e_{\eps,\ell}(\sigma)\}$,
$(\omega,x) \in h(\sigma)$, 
and the complexity of Algorithm 8.16 
in \cite{BPRbook2}, 
it is clear
that the number and degrees of the polynomials in the family ${\mathcal A}_D$
are bounded by $(s\ell m  d)^{2^{O(m+k)}}$.
We then
use the algorithm implicit in Theorem  \ref{the:triangulation} (Triangulation),
with ${\mathcal A}_{D}$ as input, to obtain the required triangulation.

The closures of the sets
\[
\{(y,\omega,x) \;\mid\; y \in c \in  {\mathcal K}(D,h^{-1}(\omega,x)), \; 
(\omega,x) \in h(h_D(\theta))\}
\] 
form a regular cell complex which we denote by
${\mathcal K}(D)$.

The following proposition gives an upper bound on the size of the
complex ${\mathcal K}(D)$. We use the notation introduced in the previous
paragraph.
\begin{proposition}
\label{pro:complexity}
For each $(\omega,x) \in h(D)$, the number of cells in 
${\mathcal K}(D,h^{-1}(\omega,x))$
is bounded by $\ell^{O(m+k)}$. Moreover, the number of cells in the complex
${\mathcal K}(D)$ is bounded by $(s\ell m d)^{2^{O(m+k)}}$.
\end{proposition}

\begin{proof}
The first part of the proposition follows from the fact that there are at
most $(\ell+1)^{m+k+1}$ 
choices for the linear space $L$ and the number of $(j-1)$
dimensional cells contained in $L$ is bounded by $2^{m+k}$ (which is an upper
bound on the number of full dimensional cells in an arrangement of at most
$m+k$ hyperplanes).
The second part is a consequence of the complexity estimate
in Theorem \ref{the:triangulation} (Triangulation) and the bounds on
number and degrees of polynomials in the family ${\mathcal A}_D$
stated above.
\end{proof}

\subsubsection{Definition of ${\mathcal K}(B,V)$}

Note that there is a homeomorphism
$$i_{D,\sigma_i}: |{\mathcal K}(\sigma_i,\tau)| \cap \varphi_1^{-1}(h(D))
\rightarrow |{\mathcal K}(D)|$$
which takes each cell of 
$|{\mathcal K}(\sigma_i,\tau)| \cap \varphi_1^{-1}(h(D))$
to a union of cells in ${\mathcal K}(D)$. 
We use these homeomorphisms 
to glue the cell complexes ${\mathcal K}(\sigma_i,\tau)$
together to form the cell complex ${\mathcal K}(B,V)$.

\begin{definition}
\label{eqn:definitionofcellsimplex}
The complex
${\mathcal K}(B,V)$ is the union of all the complexes
${\mathcal K}(D)$ constructed above, where we use the maps
$i_{D,\sigma_i}$ to make the obvious identifications. It is clear that 
${\mathcal K}(B,V)$ so defined is a regular cell complex. 
\end{definition}

We have 
\begin{proposition}
\label{pro:iso2}
$|{\mathcal K}(B,V)|$ is homotopy equivalent to $B$.
\end{proposition}

\begin{proof}
We have from Proposition \ref{pro:homotopy1} that the
semi-algebraic set $C \subset B$ (see \eqref{eqn:definition_of_C} 
for definition) is homotopy equivalent to $B$. We now prove that
$|{\mathcal K}(B,V)|$ is homotopy equivalent to $C$ which will prove the
proposition.

Let $X_{m+k} = |{\mathcal K}(B,V)|$ and for
$0 \leq j \leq m+k-1$, let $X_j = \lim_{\eps_{j}} X_{j+1}$.

It follows from an application of the
Vietoris-Smale theorem \cite{Smale} that for each $j, 0 \leq j \leq m+k-1$, 
$\E(X_j,\R\la\eps_0,\ldots,\eps_j\ra)$ is homotopy equivalent to
$X_{j+1}$. Also, by construction of ${\mathcal K}(B,V)$, we have that
$X_0 = \lim_{\eps_0} |{\mathcal K}(B,V)| = C$,
which proves the proposition.
\end{proof}

We also have
\begin{proposition}
\label{pro:complexity2}
The number of cells in the
cell complex ${\mathcal K}(B,V)$ is 
bounded by $(s\ell m d)^{2^{O(m+k)}}$.
\end{proposition}
\begin{proof}
The proposition is a consequence of Proposition \ref{pro:complexity} and
the fact that the number of cells in the complex 
${\mathcal C}(\Delta)$ is bounded by $(s\ell m d)^{2^{O(m+k)}}.$
\end{proof}

\subsubsection{Algorithm for computing the Betti numbers in the 
homogeneous
union case}

We now describe formally an algorithm for computing the Betti numbers of $A^h$
using the complex 
${\mathcal K}(B,V)$ described above.

\begin{algorithm}[Betti numbers, homogeneous union case]
\label{alg:union_betti0}
\begin{itemize}
\item[]
\item[{\sc Input}]
\item[]
\item
A family of polynomials
$
\displaystyle{
{\mathcal Q}^h \subset \R[Y_0,\ldots,Y_\ell,X_1,\ldots,X_k],
}
$
homogeneous of degree 2 in the variables $Y_0,\ldots,Y_\ell$,
$\deg_{X}(Q^h) \leq d, Q^h \in {\mathcal Q}^h, \#({\mathcal Q}^h)=m$,

\item
another family of polynomials
$
{\mathcal P} 
\subset \R[X_1,\ldots,X_k],
$
with $\deg_{X}(P) \leq d, P \in {\mathcal P}, \#({\mathcal P})=s$,
\item
a ${\mathcal P}$-closed formula $\Phi(x)$ defining a bounded
${\mathcal P}$-closed semi-algebraic set $V \subset \R^k$.

\hide{
\item
the semi-algebraic set
$$
\displaylines{
A^h = \bigcup_{Q^h \in {\mathcal Q}^h }
\{ (y,x) \;\mid\; |y|=1\; \wedge\; Q(y,x) \leq 0\; \wedge \; \Phi(x)\}.
}
$$
}
\end{itemize}

\item [{\sc Output}]
\item
\begin{itemize}
\item
a description of the cell complex 
${\mathcal K}(B,V)$,
\item
the Betti numbers of $A^h$ where
the semi-algebraic set $A^h$ is defined by
$$
\displaylines{
A^h = \bigcup_{Q^h \in {\mathcal Q}^h }
\{ (y,x) \;\mid\; |y|=1\; \wedge\; Q(y,x) \leq 0\; \wedge \; \Phi(x)\}.
}
$$
\end{itemize}

\item [{\sc Procedure}]
\item[]
\item[Step 1.]
Call Algorithm \ref{alg:triangulation} (Index Invariant Triangulation)
with input
 ${\mathcal Q}^h, {\mathcal P}$ and  $\Phi$ and compute $h$ and $\Delta$.
\item[Step 2.]
Construct the cell complex ${\mathcal C}(\Delta)$ (following
its definition given in Section \ref{subsec:cellcomplex}).
\item[Step 3.]
For each cell $D \in {\mathcal C}(\Delta)$,
compute, using 
the algorithm implicit in Theorem  \ref{the:triangulation} (Triangulation),
the cell complex ${\mathcal K}(D)$.
\item[Step 4.]
Compute a description of ${\mathcal K}(B,V)$, 
including 
the matrices corresponding to the differentials in the complex 
$\Ch_{\bullet}({\mathcal K}(B,V)) $.
\item[Step 5.] 
Compute the Betti numbers of the complex 
$\Ch_{\bullet}({\mathcal K}(B,V))$ using linear algebra.
\end{algorithm}

\vspace{.1in}
\noindent
{\sc Complexity Analysis:}
The complexity of the algorithm is 
$ 
(s\ell m  d)^{2^{O(m+k)}},
$
using the complexity of Algorithm \ref{alg:triangulation}.
\qedsymbol

\vspace{.1in}
\noindent
{\sc Proof of Correctness:}
The correctness of the algorithm is a consequence of 
the correctness of Algorithm \ref{alg:triangulation} and 
Proposition \ref{pro:iso2}.
\qedsymbol

\subsection{Computing Betti numbers in the homogeneous intersection case}

\subsubsection{Definition of ${\mathcal K}(B_I,V)$}

We now define a subcomplex of ${\mathcal K}(B,V)$ corresponding to 
a subset $I\subset [m]$.

We first extend a few definitions from Section \ref{sec:proof}.

For each subset $I \subset [m]$,
we denote by 
 ${\mathcal Q}^h_I$ the subset of ${\mathcal Q}^h$ of polynomials with indices 
in $I$ and by $\Omega_{I}$ the subset of 
$$
\Omega =
\{\omega \in \R^{m} \mid  |\omega| = 1, \omega_i \leq 0, 1 \leq i \leq m\},
$$
obtained by setting the coordinates corresponding to the elements of
$[m]\setminus I$ to $0$.
More precisely,
$$
\Omega_I = 
\{\omega \in \R^{m} \mid  |\omega| = 1, \omega_i \leq 0, \mbox{ for } i \in I, 
\mbox{ and } \omega_i = 0 \mbox{ for } i \in [m]\setminus I\}.
$$
Note that we have a natural inclusion
$\Omega_{I} \hookrightarrow \Omega_{[m]}=\Omega$.

Similarly, we denote by $F_{I} \subset F = F_{[m]}$, the set 
$\Omega_{I} \times V$, and
denote by $B_{I} \subset \Omega_{I} 
\times \Sphere^{\ell} \times V$ 
the semi-algebraic set defined by
\[
B_{I} = \{ (\omega,y,x)\mid \omega \in \Omega_{I}, y\in \Sphere^{\ell}, x \in 
V,  \; \la \omega,{\mathcal Q}^h \ra (y,x) \geq 0\}.
\]

We denote by $\varphi_{1,I}: B_{I} \rightarrow 
F_{I}$ and 
$\varphi_{2,I}: B_{I} \rightarrow \Sphere^{\ell} \times V$ the two projection 
maps.

Now we define  ${\mathcal K}(B_I,V)$ for every $I\subset [m]$.

\begin{definition}
\label{eqn:definitionofcellsimplexI}
The complex
${\mathcal K}(B_I,V)$ is the union of all the complexes
${\mathcal K}(D)$  in ${\mathcal C}(\Delta_I)$  , where ${\mathcal C}(\Delta_I)$ is the subcomplex of
 ${\mathcal C}(\Delta)$ consisting of cells contained in 
$\Delta_I=h^{-1}(F_I)$.
\end{definition}

Using proofs similar to the ones give for ${\mathcal K}(B,V)$, we have
\begin{proposition}
\label{pro:iso2I}
$|{\mathcal K}(B_I,V)|$ is homotopy equivalent to $B_I$. \qed
\end{proposition}

\begin{algorithm}[Computing the collection of ${\mathcal K}(B_I,V)$,
$I \subset \lbrack 1\ldots,m \rbrack$]
\label{alg:union_betti}
\item[]
\item[{\sc Input}]
\item[]
\begin{itemize}
\item
$
\displaystyle{
{\mathcal Q}^h = \{Q_1^h,\ldots,Q_m^h\} \subset \R[Y_0,\ldots,Y_\ell,X_1,\ldots,X_k],
}
$
where each $Q_i^h$ is homogeneous of degree $2$ in the variables $Y_0,\ldots,Y_\ell$,
and of degree at most $d$ in $X_1,\ldots,X_k$,
\item
$
{\mathcal P} 
\subset \R[X_1,\ldots,X_k],
$
with $\deg(P) \leq d,  P \in {\mathcal P}$,
\item
a ${\mathcal P}$-closed formula $\Phi(x)$ defining a bounded
${\mathcal P}$-closed semi-algebraic set $V \subset \R^k$.
\end{itemize}

\item [{\sc Output}]
\item[]
\item
\begin{itemize}
\item
For each subset $I \subset [m]$,
a description of the cell complex 
${\mathcal K}(B_I,V)$.
\item
For each $I \subset J \subset [m]$, a homomorphism
\[
i_{I,J}: \Ch_{\bullet}(B_I,V) \rightarrow \Ch_\bullet(B_J,V)
\]
inducing the inclusion  homomorphism
$i_{I,J *}: \HH_*(B_I,V) \rightarrow \HH_*(B_J,V)$.
\end{itemize}
\item [{\sc Procedure}]
\item[]
\item[Step 1.]
Call Algorithm \ref{alg:union_betti0} to compute ${\mathcal K}(B,V)$.
\item[Step 2.]
Give a description of ${\mathcal K}(B_I,V)$ for each $I \subset [m]$
and compute the matrices corresponding to the differentials in the complex 
$\Ch_{\bullet}({\mathcal K}(B_I,V)) $.
\item[Step 3.]
For
$I \subset J\subset [m]$ with
compute the matrices for the homomorphisms of complexes,
$$
{i}_{I,J}: \Ch_{\bullet}({\mathcal K}(B_I,V)) \rightarrow \Ch_\bullet({\mathcal K}(B_J,V))
$$
in the following way.

The complex ${\mathcal K}(B_I,V)$ 
is a subcomplex of 
${\mathcal K}(B_J,V)$ by construction.
Compute the matrix for the inclusion  homomorphism,
$$
i_{I,J}:
\Ch_{\bullet}({\mathcal K}(B_I,V))
\rightarrow 
\Ch_{\bullet}({\mathcal K}(B_J,V))
$$ 
and output the matrix for the homomorphism. 
\end{algorithm}

\vspace{.1in}
\noindent
{\sc Complexity Analysis:}
The complexity of the algorithm is 
$ 
(s\ell m  d)^{2^{O(m+k)}},
$
using the complexity of Algorithm \ref{alg:triangulation}.
\qedsymbol

\vspace{.1in}
\noindent
{\sc Proof of Correctness:}
The correctness of the algorithm is a consequence of 
the correctness of Algorithm \ref{alg:triangulation} and 
Proposition \ref{pro:iso2}.
\qedsymbol

\subsubsection{Algorithm for 
computing the Betti numbers in the
homogeneous
intersection case}

Let
$W^h \subset \Sphere^{\ell} \times \R^k$ be the semi-algebraic set defined by
$$
\displaylines{
W^h = \bigcap_{Q \in {\mathcal Q}^h}
\{ (y,x) \;\mid\; \vert y \vert =1 \wedge Q(y,x) \leq 0\; \wedge \; \Phi(x)\},
}
$$
using Notation \ref{not:AhWh}.

Then 
\begin{equation}
\label{eqn:bicomplex}
\HH_*(W^h) \cong \HH^*(\Tot_\bullet(\NN_{\bullet,\bullet}({\mathcal K}(B,V)))),
\end{equation}

where $\NN_{\bullet,\bullet}({\mathcal K}(B,V))$ is the bi-complex
\begin{equation}\label{eqn:bicompNpq}
\NN_{p,q}({\mathcal K}(B,V)) = \bigoplus_{J \subset [m],\#(J) = p+1} \Ch_q({\mathcal K}(B_J,V)), 
\end{equation}
with the horizontal and vertical differentials defined as follows.
The vertical differentials
\begin{equation}\label{eqn:vertdiff}
d_{p,q}: \NN_{p,q}({\mathcal K}(B,V)) \rightarrow \NN_{p,q-1}({\mathcal K}(B,V)),
\end{equation}
are induced by the boundary homomorphisms,
\begin{equation*}
\partial_q: \Ch_q({\mathcal K}(B_I,V)) \rightarrow \Ch_{q-1}({\mathcal K}(B_J,V)), 
\end{equation*}
and the horizontal differentials
\begin{equation}\label{eqn:hordiff}
\delta_{p,q}: \NN_{p,q}({\mathcal K}(B,V)) \rightarrow \NN_{p+1,q}({\mathcal K}(B,V))
\end{equation}
are defined by
$$
\displaylines{
(\delta_{p,q}(\varphi))_{J} =
\sum_{j \in J} i_{J\setminus \{j\},J}(\varphi_{J\setminus \{j\}}),
}
$$
where $J \subset [m], \#(J)=p+1$, 
$$
\displaylines{
\varphi \in \NN_{p,q}({\mathcal K}(B,V)) = \bigoplus_{J \subset [m],\#(J) = p+1} \Ch_\bullet({\mathcal K}(B_J,V)),
}
$$
and for $I \subset J \subset [m]$
\[
i_{I,J}: \Ch_{\bullet}({\mathcal K}(B_I,V)) \rightarrow \Ch_\bullet({\mathcal K}(B_J,V))
\] 
denotes the homomorphism induced by inclusion.

For a proof of (\ref{eqn:bicomplex}) see \cite{Bas05-top}.

Using (\ref{eqn:bicomplex}),  we are able to compute the Betti numbers of 
$W^h$ 
using only linear algebra,
once we have computed 
the various complexes ${\mathcal K}(B_I,V)$, as well as the homomorphisms
$i_{I,J}$ for all $ I\subset J \subset [m]$ using Algorithm 
\ref{alg:union_betti}. Moreover, the complexity of this algorithm is 
asymptotically the same as that of Algorithm \ref{alg:union_betti}.

We now formally describe this algorithm.

\begin{algorithm}[Betti numbers, homogeneous intersection case]
\label{alg:basic_betti}
\item[]
\item[{\sc Input}]
\item[]
\begin{itemize}
\item
A family of polynomials,
$
{\mathcal Q}^h =\{Q_1^h,\ldots,Q_m^h\} \subset  \R[Y_0,\ldots,Y_\ell,X_1,\ldots,X_k],
$
homogeneous of degree 2 with respect to 
$Y_0,\ldots,Y_\ell$,
 $
\deg_{X}(Q^h) \leq d, 
Q^h \in {\mathcal Q}^h$, 
\item
another family, 
${\mathcal P} 
\subset \R[X_1,\ldots,X_k]$ with
$ \deg_{X}(P) \leq d, P \in {\mathcal P}, \#({\mathcal P})=s$,

\item
a formula $\Phi$ defining a bounded 
${\mathcal P}$-closed semi-algebraic set $V$.

\hide{
\item the semi-algebraic set
$$
\displaylines{
W^h = \bigcap_{Q^h\in {\mathcal Q}^h}
\{ (y,x) \;\mid\; |y|=1\; \wedge\; Q^h(y,x) \leq 0\; \wedge \; \Phi(x)\}.
}
$$
}
\end{itemize}

\item [{\sc Output}] the Betti numbers $b_i(W^h)$,
where $W^h$ is
the semi-algebraic set defind by
$$
\displaylines{
W^h = \bigcap_{Q^h\in {\mathcal Q}^h}
\{ (y,x) \;\mid\; |y|=1\; \wedge\; Q^h(y,x) \leq 0\; \wedge \; \Phi(x)\}.
}
$$

\item [{\sc Procedure}]
\item[]
\item[Step 1.]
Call Algorithm \ref{alg:union_betti} 
(Computing the collection of ${\mathcal K}(B_I,V)$) to
compute for each $I \subset J \subset [m]$, the complex
$\Ch_\bullet({\mathcal K}(B_I,V))$ using the natural basis consisting of the cells
of ${\mathcal K}(B_I,V)$ of various dimensions, as well as the matrices in this basis 
for the  inclusion homomorphisms
$$
\displaylines{
i_{I,J}:\Ch_\bullet({\mathcal K}(B_I,V)) \rightarrow 
\Ch_\bullet({\mathcal K}(B_J,V)).
}
$$
\item[Step 2.]
Using the data from the previous step,
compute matrices corresponding to the differentials in the complex,
$\Tot_\bullet(\NN_{\bullet,\bullet}({\mathcal K}(B,V)))$,
where $\NN_{\bullet,\bullet}({\mathcal K}(B,V))$ is the bi-complex described by
\eqref{eqn:bicompNpq}-\eqref{eqn:hordiff}.
\item [Step 3.]
Compute, using linear algebra subroutines
$$
\displaylines{
b_i(W^h) = \HH_i((\Tot_\bullet(\NN_{\bullet,\bullet}({\mathcal K}(B,V))))).
}
$$
\end{algorithm}

\vspace{.1in}
\noindent
{\sc Complexity Analysis:}
The complexity of the algorithm is dominated by the first step, whose complexity is
$ 
(s\ell m d)^{2^{O(m+k)}},
$
using the complexity of Algorithm \ref{alg:union_betti}.
\qedsymbol

\vspace{.1in}
\noindent
{\sc Proof of Correctness:}
The correctness of the algorithm is a consequence of 
the correctness of Algorithm \ref{alg:union_betti} and 
\eqref{eqn:bicomplex}.
\qedsymbol

\subsection{Computing Betti numbers of general ${\mathcal P}\cup {\mathcal Q}$-closed 
sets}
Let $S \subset  \R^{\ell+k}$ 
be a  semi-algebraic set 
defined by a ${\mathcal P}\cup {\mathcal Q}$
closed formula $\Phi$. 

Let $\bar\Sigma_{{\mathcal Q}}$ denote the set of all possible weak 
sign conditions on the family ${\mathcal Q}$, i.e.
\[
\bar\Sigma_{{\mathcal Q}} = \{0, \{0,1\},\{0,-1\}\}^{{\mathcal Q}}.
\]

In the last section we defined a bi-complex 
$\NN_{\bullet,\bullet}({\mathcal K}(B,V))$ whose total complex has
homology groups isomorphic to these of the semi-algebraic set 
$W^h = \RR(\rho^h\cap\phi)$, where 
$\rho \in \bar\Sigma_{{\mathcal Q}}$ is given by 
${\rho}(Q_i) = \{0,-1\}$ for each $i, 1\leq i \leq m$, and
$\rho^h$ is obtained from $\rho$ by replacing 
each $Q_i \in {\mathcal Q}$ by ${\mathcal Q_i}^h$.
We now generalize this definition to the case
of multiple weak sign conditions. More precisely,
given a set 
$\Sigma = 
\{\rho_1,\ldots,\rho_N\} \subset \bar\Sigma_{{\mathcal Q}}$, 
we define a corresponding bi-complex having properties similar to that
of $\NN_{\bullet,\bullet}({\mathcal K}(B,V))$, but now with respect
to $\Sigma$ instead of a single weak sign condition $\rho$.

Without loss of generality we can write $\Phi$ in the 
form
\[
\Phi = \bigvee_{\rho \in \bar\Sigma_{{\mathcal Q}}} \rho \wedge \phi_\rho,
\]
where each $\phi_\rho$ is a ${\mathcal P}$-closed formula.

Let 
\begin{align*}
W_\rho =& \RR(\rho \wedge \phi_\rho, \R^{\ell+k}), \\
V_\rho =& \RR(\phi_\rho,\R^k).
\end{align*}

Let $1 \gg \eps > 0$ be an infinitesimal, and let

\begin{align*}
Q_0 =& \eps^2(Y_1^2 + \cdots + Y_\ell^2) - 1,\\
P_0 =& \eps^2(X_1^2 + \cdots + X_k^2) - 1,
\end{align*}

and
$S_b \subset \R\la\eps\ra^{\ell+k}$ be the
semi-algebraic set defined by
$$
\displaylines{
S_b = \bigcap_{i=0}^{m}
\{ (y,x) \;\mid\;  Q_0(y) \leq 0\; \wedge \;P_0(x) \leq 0
\; \wedge \; \Phi(x)\}.
}
$$

We denote by $\Phi_b$ (resp. $\phi_{\rho,b}$) the formula
$(Q_0(y) \leq 0) \;\wedge\; (P_0(x) \leq 0)\; \wedge \; \Phi$ (resp. 
$(P_0(x) \leq 0)\; \wedge \; \phi_\rho$.)

Let $S_b^h, W_{\rho,b}^h\subset 
\Sphere^{\ell} \times \R\la\eps\ra^k$ be the sets defined
by $\Phi_b$ and $\rho\; \wedge \; (Q_0^h \leq 0)\;\wedge \;\phi_{\rho,b}$ 
respectively
on $\Sphere^{\ell} \times \R\la\eps\ra^k$  after 
replacing each $Q_i \in {\mathcal Q}$ by $Q_i^h$ 
in the formulas $\Phi_b$ and $\rho$.

Let 
\[
V_{\rho,b}= \RR(\phi_{\rho,b},\R\la\eps\ra^k).
\]
Let ${\mathcal Q}^{h}_{\pm} = \{ \pm Q^h \;\mid\; Q^h \in {\mathcal Q}^h \}$,
and let ${\mathcal K}(B,V)$ denote the complex constructed by Algorithm 
\ref{alg:union_betti}, with input the families of
polynomials ${\mathcal Q}^{h}_{\pm}$, 
${\mathcal P}_b = {\mathcal P} \cup \{P_0\}$,
and the semi-algebraic subset $V = B_k(0,1/\eps)$.

It follows from the correctness of Algorithm 
\ref{alg:basic_betti} that for each $\rho \in \bar\Sigma_{{\mathcal Q}}$,
there exists 
$J_\rho \subset {\mathcal Q}^{h}_{\pm}$, and 
a subcomplex, ${\mathcal K}(B_{J_\rho},V_{\rho,b}) \subset 
{\mathcal K}(B,V)$,
such that 
\[
\HH_*(\Tot_\bullet(\NN_{\bullet,\bullet}({\mathcal K}(B_{J_\rho},V_{\rho,b})))
\cong
\HH_*(W_{\rho,b}^h).
\]

More generally, for any
$\Sigma = 
\{\rho_1,\ldots,\rho_N\} \subset \bar\Sigma_{{\mathcal Q}}$, 
there exists a subcomplex, ${\mathcal K}(B_{J_{\rho}},V_{\rho_1,b} \cap \cdots
V_{\rho_N,b}) \subset 
{\mathcal K}(B,V)$,
such that the homology groups of the complex
\begin{equation}
\label{eqn:defofch}
\Ch_{\Sigma,\bullet} = 
\Tot_\bullet(\NN_{\bullet,\bullet}({\mathcal K}(B_{J_{\rho}},V_{\rho_1,b} \cap \cdots
V_{\rho_N,b}))
\end{equation}
are naturally isomorphic to those of $W_{\rho,b}^h$, 
where $\rho$ is the common refinement of $\rho_1,\ldots,\rho_N$
defined by
\begin{equation}
\label{eqn:defofrefinement}
\rho(P) = \bigcap_{i=0}^{N}\rho_i(P)
\end{equation}
for each $P \in {\mathcal A}$.
Moreover, for $\Sigma \subset \Sigma' \subset  \bar\Sigma_{{\mathcal Q}}$, 
there exists a natural homomorphism,
\[
i_{\Sigma,\Sigma'} : \Ch_{\Sigma',\bullet} \rightarrow \Ch_{\Sigma,\bullet}
\]
such that the induced homomorphism,
\[
i_{\Sigma,\Sigma',*} : \HH_*(\Ch_{\Sigma',\bullet})\rightarrow 
\HH_*(\Ch_{\Sigma,\bullet})
\]
is the one induced by the inclusion
\[
\bigcap_{\rho \in \Sigma'} W_{\rho,b}^h \hookrightarrow \bigcap_{\rho \in 
\Sigma} W_{\rho,b}^h.
\]

\begin{definition}
Let $\Ch_{\bullet}(\Phi)$ denote the complex defined by
\begin{equation}
\label{eqn:defofChphi}
\Ch_{\bullet}(\Phi) = \Tot_{\bullet}(\NN_{\bullet,\bullet}(\Phi)), 
\end{equation}
where
\begin{equation}
\label{eqn:defofNN}
\NN_{p,q}(\Phi) = \bigoplus_{\Sigma \subset \bar\Sigma_{{\mathcal Q}},\#(\Sigma) = p+1} 
\Ch_{\Sigma,q}.
\end{equation}
The vertical and horizontal homomorphisms in the complex
$\NN_{\bullet,\bullet}(\Phi)$ are induced by the the differentials in the
individual complexes  $\Ch_{\Sigma,\bullet}$ and the inclusion homomorphisms
$i_{\Sigma,\Sigma'}$ respectively.
\end{definition}

By the properties of the complexes $\Ch_{\Sigma,\bullet}$ stated
above and the exactness of the generalized Mayer-Vietoris sequence, we
obtain
\begin{theorem}
\label{the:bettigeneral}
\[
\HH_*(S_b^h) \cong \HH_*(\Ch_{\bullet}(\Phi)).
\]
\end{theorem}

We are now in a position to describe formally 
the algorithm for computing all the
Betti numbers of a given ${\mathcal P}\cup{\mathcal Q}$-closed set $S$.

\subsubsection{Description of the algorithm in the general case}
\begin{algorithm} [Betti numbers, general case]
\label{alg:general_betti}
\item[{\sc Input}]
\item[]
\begin{itemize}
\item A family of polynomials ${\mathcal Q} = \{Q_1,\ldots,Q_m\} \subset \R[Y_1,\ldots,Y_\ell,X_1,\ldots,X_k]$,
with
$
\deg_Y(Q_i) \leq 2, \deg_X(Q_i) \leq d, 1 \leq i \leq \ell,
$
\item
another family of polynomials ${\mathcal P} 
\subset \R[X_1,\ldots,X_k]$
with 
$\deg(P) \leq d,  P \in {\mathcal P}$,
\item a ${\mathcal Q} \cup {\mathcal P}$-closed semi-algebraic set $S$
defined by a ${\mathcal Q} \cup {\mathcal P}$-closed formula $\Phi$.

\end{itemize}
\item [{\sc Output}]
the Betti numbers $b_0(S),\ldots,b_{k+\ell-1}(S)$.

\item [{\sc Procedure}]
\item[]
\item[Step 1.]
Define  $Q_0=\eps_0^2(Y_1^2+\ldots+Y_\ell^2)-1$, $P_0=\eps_0^2(X_1^2+\ldots+X_k^2)-1$.
Replace 
$S$ by $\RR(S,\R \la \eps \ra)\cap (\RR(P_0 \le 0) \times \RR(Q_0 \le 0))$.

\item[Step 2.]
Define ${\mathcal Q}^{h}_{\pm} = \{ \pm Q^h \;\mid\; Q^h \in 
{\mathcal Q}^h \} \cup \{Q_0^h\}$,
and let ${\mathcal K}(B,V)$ denote the complex constructed by Algorithm 
\ref{alg:union_betti}, with input the families of
polynomials ${\mathcal Q}^{h}_{\pm}, {\mathcal P}_b$,
and the semi-algebraic set $V = B_k(0,1/\eps) \subset \R\la\eps\ra^k$.

\item[Step 3.]
Compute, using the definitions given above, the matrices corresponding to the
differentials in the complex
$\Ch_{\bullet}(\Phi)$.

\item[Step 4.]
Compute, using linear algebra subroutines, for each $i, 0 \leq i \leq k+\ell-1$ 
$$
\displaylines{
b_i(S_b^h) = \HH_i(\Ch_{\bullet}(\Phi)).
}
$$
\item [Step 5.]
Output for each $i, 0 \leq i \leq k+\ell-1$, 
\[
b_i(S) =  \frac{1}{2} b_i(S_b^h).
\]
\end{algorithm}
\noindent\textsc
{Proof of Correctness:}
The correctness of the algorithm follows from Theorem \ref{the:bettigeneral}
and the correctness of Algorithm \ref{alg:union_betti}.
\qedsymbol

\noindent\textsc
{Complexity Analysis:}
Since $\#(\bar\Sigma_{{\mathcal Q}}) = 3^m$, the number of subsets
that enters in the definition of $\NN_{\bullet,\bullet}(\Phi)$ 
(cf.  \eqref{eqn:defofNN}) is at most
$2^{3^m}$.
The complexity of the algorithm is now seen to be
$ 
(s\ell m d)^{2^{O(m+k)}},
$
using the complexity of Algorithm \ref{alg:union_betti}.
\qedsymbol

\begin{proof}[Proof of Theorem \ref{the:algo-betti}]
The correctness and complexity analysis of Agorithm \ref{alg:general_betti}
also proves Theorem \ref{the:algo-betti}.
\end{proof}

\bibliographystyle{amsplain}
\bibliography{master}

\end{document}